\documentclass[12pt]{article}
\usepackage{mathpazo}
\usepackage{latexsym}
\usepackage{a4wide}
\usepackage{mathrsfs}
\usepackage{amsmath}
\usepackage{amssymb}
\usepackage{comment}

\newtheorem{thm}{Theorem}[section]

\newtheorem{cor}[thm]{Corollary}

\newtheorem{exl}[thm]{Example}

\newtheorem{rem}[thm]{Remark}
\newtheorem{rems}[thm]{Remark}
\newtheorem{lemma}[thm]{Lemma}
\newtheorem{prop}[thm]{Proposition}

\newcommand{\N}{\mathbb{N}}
\newcommand{\R}{\mathbb{R}}
\newcommand{\iii}{\mathtt{i}}
\newcommand{\jjj}{\mathtt{j}}

\numberwithin{equation}{section}
\newenvironment{proof}{\vspace{-5pt}\noindent\textbf{Proof:}}{$\Box$\\}

\begin{document}
\title{The doubling metric and doubling measures\thanks{We would like to thank Natalia Chernova, Marcus Pivato, and John Levy for helpful discussions. VS has been partly supported by the Academy of Finland via the Centre of Excellence in Analysis and Dynamics Research. We are particularly indebted to an anonymous reviewer for his exceptionally detailed and helpful report.}}

\author{J\'{a}nos Flesch\footnote{Department of Quantitative Economics, Maastricht University, P.O.Box 616, 6200 MD, The Netherlands. E-mail: \texttt{j.flesch@maastrichtuniversity.nl}.} \and Arkadi Predtetchinski\footnote{Department of Economics, Maastricht University, P.O.Box 616, 6200 MD, The Netherlands. E-mail: \texttt{a.predtetchinski@maastrichtuniversity.nl.}} \and Ville Suomala\footnote{Department of Mathematical sciences, University of Oulu, P.O Box 3000, FI--90014, Finland. E-mail: \texttt{ville.suomala@oulu.fi}}}
\maketitle

\begin{abstract}
\noindent
We introduce the so-called doubling metric on the collection of non-empty bounded open subsets of a metric space. Given an open subset $U$ of a metric space $X$, the predecessor $U_{*}$ of $U$ is defined by doubling the radii of all open balls contained inside $U$, and taking their union. The predecessor of $U$ is an open set containing $U$. The directed doubling distance between $U$ and another subset $V$ is the number of times that the predecessor operation needs to be applied to $U$ to obtain a set that contains $V$. Finally, the doubling distance between open sets $U$ and $V$ is the maximum of the directed distance between $U$ and $V$ and the directed distance between $V$ and $U$.
\end{abstract}


\section{Introduction}
There are many ways to define a distance between two subsets $U$ and $V$ of a metric space $(X,\delta)$. For instance, one may consider the minimal distance between pairs of points $(x,y)\in U\times V$, that is $\inf_{x\in U\,,y\in V}\delta(x,y)$. Obviously, this distance is not usually a metric on any reasonable family of subsets of $X$ since it equals zero whenever the closures of $U$ and $V$ have points in common. Thus, if we want to define a distance between subsets of $X$ in such a way that it satisfies the axioms of a metric, we need
to consider more subtle definitions. In this respect, the Hausdorff distance,
\[d_H(U,V)=\max\{\sup_{x\in U} \inf_{y\in V}\delta(x,y),\sup_{y\in V} \inf_{x\in U}\delta(x,y)\}\,,\]
is one of the most classical and most used concepts. As is well-known, it defines a metric on the collection of non-empty compact subsets of $X$. Moreover, the space of all non-empty compact subsets of $X$ with the metric $d_H$ inherits many topological properties (such as compactness or completness) from the space $X$. The applications of the Hausdorff distance and its many variants are numerous and far-reaching ranging from topology and geometry to computer vision (see e.g. \cite{Arutyunovetal2016,
	%Falconer1986,
	Munkres2000,
	%Buragoetal2001,
	Berindeetal2013}).

The Hausdorff distance is very robust with respect to discretizations. If $X$ is separable, then for any subset $S\subseteq X$ and any $\epsilon$, there is a countable set $S_\epsilon\subseteq S$ that is $\epsilon$-close to $S$ in the Hausdorff distance. Passing to such a discretization can be a huge advantage, but, from the point of view of analysis, this causes limitations for the usage of the Hausdorff distance.
For instance, if the space $X$ is endowed with a measure $\mu$, there is in general no way to relate the measures of $U$ and $V$ based on the Hausdorff distance of $U$ and $V$ alone, no matter how regular the measure $\mu$ is.

In the present paper, we introduce the so-called doubling metric on the collection of non-empty bounded open subsets of a metric space. The definition is based on an intuitive idea of a predecessor of a set. Given an open subset $U$ of a metric space $X$, the predecessor $U_{*}$ of $U$ is obtained by doubling the radii of all open balls contained inside $U$, and taking their union. The predecessor of $U$ is an open set and $U\subseteq U_*$. The directed doubling distance from $U$ to another set $V$ is the number of times that the predecessor operation has to be applied to $U$ to obtain a set that contains $V$. Finally, the (symmetric) doubling distance between $U$ and $V$ is the maximum of the directed distance from $U$ to $V$ and the directed distance from $V$ to $U$. The directed and symmetric distances satisfy the triangle inequality on all open subsets of $X$, but need not be finite. Restricting the doubling distance to non-empty bounded open sets, we obtain a genuine metric.

The doubling distance $d$ has a number of important features in terms of the fine structure of the sets $U$ and $V$. Firstly, it is invariant under similarity transformations (see Lemma \ref{thm.union}). Secondly, for any doubling measure, the measures of two open sets $U,V\subseteq X$ that are close in the doubling distance are comparable. More precisely, it holds that
\begin{equation}\label{eq:mled}
\mu(U)\ge C^{-3 d(U,V)}\mu(V)\,,
\end{equation}
whenever $\mu$ is a $C$-doubling measure (Theorem \ref{7Rcor}). Recall that a measure $\mu$ on $X$ is $C$-doubling, if
\begin{equation}\label{eqn.doubling}
0<\mu(B(x,2r)) \leq C \mu(B(x,r)) < \infty.
\end{equation}
for all $x\in X$ and $r>0$. Here $B(x,r)$ denotes the closed ball
$B(x,r)=\{y\in X\,:\,\delta(y,x)\le r\}$. Doubling measures play an essential role in modern analysis and the existence  of doubling measures supported by $X$ is an important regularity feature of the space $X$. It is well known (see \cite{LuukkainenSaksman1998}) that metric spaces that are complete and geometrically doubling support some doubling measures. Without the completeness assumption, the existence of doubling measures is a subtle question, see \cite{KaufmanWu1995, Saksmann1999, CsornyeiSuomala2012}. Recall that a metric space $X$ is geometrically doubling, if there is an $N\in\N$ such that each ball $B\subseteq X$ may be covered by $N$ balls with half the radius of $B$.

Our main result (Theorem \ref{thm.simple}) states that the implication in \eqref{eq:mled} can be reversed in the case of the so-called simple open sets: a simple open set is a finite union of open balls. More precisely, the result claims the following: consider simple open sets $U$ and $V$ in a metric space $X$, and suppose that $X$ carries some doubling measures. If there is a constant $K<\infty$ such that
\[C^{-K} \mu(U)\le \mu(V)\le C^K\mu(U)\,,\]
for all $C<\infty$ and all $C$-doubling measures $\mu$ on $X$, then $d(U,V) \le 24 K+8$.

The paper is organised as follows: In Section \ref{sec:dmetric}, we introduce the doubling metric along with its basic properties. We also present an equivalent game theoretic definition, which could be of independent interest. In spaces that carry doubling measures, we define a measure variant $m$ of the doubling distance and prove the basic estimate \eqref{eq:mled} by comparing these two distances. Section \ref{sec:main} is devoted to the main result providing a sufficient condition for the comparability of the doubling distance $d$ and its measure variant $m$. In Section \ref{sec:maps}, we study continuous functions between metric spaces $X$ and $Y$ that distort the distances $d$ and/or $m$ in a Lipschitz manner. Such functions relate naturally to maps that preserve doubling measures, quantitatively, and provide a connection to quasisymmetric maps. Finally, in Section \ref{sec:remarks}, we discuss an open problem and a definition of porosity based on the doubling metric.

\section{The doubling  metric}\label{sec:dmetric}
\subsection{Main definitions}
Consider a metric space $(X, \delta)$. By a measure on $X$ we always mean a non-negative countably additive set function on the sigma-algebra of Borel subsets of $X$. Let $1\le C<\infty$. A measure $\mu$ on $X$ is said to be \emph{$C$-doubling} if it satisfies \eqref{eqn.doubling}. Equivalently, $\mu$ is $C$-doubling if for each $x \in X$ and $r > 0$ it holds that
\begin{equation}\label{eqn.doubling2}
\mu(O(x,2r)) \leq C\mu(O(x,r)) < \infty\,,
\end{equation}
where $O(x,r)=\{y\in X\,:\,\delta(y,x)<r\}$ denotes the open ball of radius $r$ centered at $x$. (The two conditions are equivalent because an open ball is a limit of an increasing sequence of closed balls, and a closed ball is a limit of a decreasing sequence of closed balls.) We call $\mu$ a \emph{doubling measure}, if it is $C$-doubling for some $C<\infty$. We let $\mathcal{D}(X)$ (resp. $\mathcal{D}_{C}(X)$) be the collection of all doubling (resp. $C$-doubling) measures on $X$.

Throughout the paper, $\mathbb{N} = \{0,1,2,\dots\}$.

We proceed with the definition of the \emph{predecessor operation}, and the \emph{directed distance}. For an open subset $U$ of $X$ define the \textit{predecessor} of $U$ to be the set
\[U_{*} = \bigcup\left\{O(x,2r):x \in X\text{ and }r > 0\text{ such that }O(x,r) \subseteq U\right\}.\]

Consider also a set $V \subseteq X$. Let $U_{*}^{0} = U$ and for each $n \in \mathbb{N}$ let $U_{*}^{n + 1} = (U_{*}^{n})_{*}$. Let $d_{\rightarrow}(U,V)$ denote the smallest number $n$ such that $V \subseteq U_{*}^{n}$. If no such number exists, we set $d_{\rightarrow}(U,V) = +\infty$. The number $d_{\rightarrow}(U,V)$ is called the \textit{directed (doubling) distance} between $U$ and $V$. We state the following obvious properties of $d_{\rightarrow}$ for future reference.

\begin{lemma}\label{thm.union}
\begin{enumerate}
\renewcommand{\labelenumi}{(\arabic{enumi})}
\renewcommand{\theenumi}{\arabic{enumi}}	
\item\label{incl} Let $U$ be an open subset of $X$. Then $U_{*}$ is an open set and $U \subseteq U_{*}$.
\item\label{u1} Let $U$ and $W$ be open subsets $X$. If $U \subseteq W$, then $U_{*} \subseteq W_{*}$.
\item\label{u2}  Let $\{U_{i}: i\in I\}$ be a collection of open subsets of $X$ and $\{V_{i}: i\in I\}$ a collection of arbitrary subsets of $X$. Then $d_{\rightarrow}(\bigcup_{i \in I} U_{i}, \bigcup_{i \in I}V_{i}) \leq \sup_{i \in I} d_{\rightarrow}(U_{i},V_{i})$.
\item\label{finite} If $U$ is a non-empty open subset of $X$ and if $V$ is a bounded subset of $X$, then $d_{\rightarrow}(U,V)$ is finite.
\item\label{u3} The directed doubling distance satisfies the triangle inequality: if $U$, $W$, and $V$ are open subsets of $X$, then $d_{\rightarrow}(U,W) \leq d_{\rightarrow}(U,V) + d_{\rightarrow}(V,W)$.
\item\label{u4} The directed doubling distance is invariant under similarity transformations. More generally, if $f\colon X\to Y$ is a bijection\footnote{Note that we use the symbols $\delta$, $d_\rightarrow$ to denote the metric/directed doubling distance on both $X$ and $Y$.} such that
\[K_1\le \frac{\delta(f(x),f(y))}{\delta(x,y)}\le K_2\]
for all $x,y\in X$, then
\[\frac{1}{K}\le \frac{d_\rightarrow(f(U),f(V))}{d_\rightarrow(U,V)}\le K\,,\]
for all open subsets $U$ and $V$ of $X$, where $K$ is the smallest natural number with $K \geq 1+ \log_2\frac{K_2}{K_1}$.
\end{enumerate}
\end{lemma}
\begin{proof}
Items \eqref{incl}-\eqref{finite} are obvious.

We prove the statement $\eqref{u3}$. We can assume that the right-hand side of the inequality is finite, for otherwise there is nothing to prove. Thus, let $d_{\rightarrow}(U,V) = n$ and $d_{\rightarrow}(V,W) = m$. Then $V \subseteq U_{*}^{n}$ and $W \subseteq V_{*}^{m}$. Using \eqref{u1}, we see that $V_{*}^{m} \subseteq U_{*}^{n + m}$ and hence $W \subseteq U_{*}^{n + m}$. Hence $d_{\rightarrow}(U,W) \leq m + n$.

Finally, to prove \eqref{u4}, we make the following observation: For any $x\in X$, $r>0$, we have
\[O(f(x),K_1 r)\subseteq f(O(x,r)) \subseteq f(O(x,2r))\subseteq O(f(x), K_2 2 r)\,,\]
implying that $d_\rightarrow(f(O(x,r)),f(O(x,2r))\le K$. Since
\begin{align*}
f(U)&=\bigcup\{f(O(x,r))\,:\,O(x,r)\subseteq U\}\,,\\
f(U_*)&=\bigcup\{f(O(x,2r))\,:\,O(x,r)\subseteq U\}.
\end{align*}
This together with \eqref{u2} implies that $d_\rightarrow(f(U),f(U_*))\le K$. Applying this estimate inductively, we find that if $V\subseteq U_*^n$, then $d_\rightarrow(f(U), f(V)) \le Kn$. This shows the inequality $d_\rightarrow(f(U),f(V)) \le K d_\rightarrow(U,V)$. The lower bound $d_\rightarrow(f(U),f(V)) \ge K^{-1} d_\rightarrow(U,V)$ follows by applying the result already obtained to $f^{-1}$.
\end{proof}

\begin{rem}\rm
\noindent\emph{(a)}	If $m\in\N$ and
\[V\subseteq\bigcup\{O(x,2^m r)\,:\,O(x,r)\subseteq U\}\,,\]
then it is obvious that $d_\rightarrow(U,V)\le m$. Usually, this implication cannot be reversed: For instance, consider $V=(-2,2)\subset\R$ and for any $m\in\N$, let $U=(-1,1)\setminus\{k2^{-m}\,:\,k \in \mathbb{Z}\}$. Then $d_\rightarrow(U,V)=2$, but $V$ is not contained in the set $\bigcup\{O(x,2^m r)\,:\,O(x,r)\subseteq U\}$.\medskip

\noindent\emph{(b)} There is nothing special about the constant $2$ in the definition of $U_*$. Instead of doubling the radii of all balls contained in $U$, we could multiply them by any given number $>1$. The distance obtained this way would be bi-Lipschitz equivalent to the doubling distance.
\end{rem}

Next we describe an equivalent game-theoretic definition of the directed distance. We mention this equivalent definition here, because we feel that the game-theoretic approach might lead to other, perhaps more subtle, definitions of a metric.

Let $U$ be an open subset $X$, an $n \geq 1$, and a point $y_{0} \in X$ be given. Consider the following $n$-stage game, denoted $\Gamma_{n}(y_{0})$:
\begin{center}
\begin{tabular}{cccccccc}
I&$x_{0},r_{0}$&&$x_{1},r_{1}$&$\cdots$&$x_{n-1},r_{n-1}$\\
II&&$y_{1}$&&$y_{2}$&$\cdots$&$y_{n}$
\end{tabular}
\end{center}
For each $i \in \{0,\dots,n - 1\}$, the move $(x_{i},r_{i}) \in X \times \mathbb{R}$ of player I is required to satisfy $\delta(y_{i},x_{i}) < 2r_{i}$; the move $y_{i + 1} \in X$ of player II is required to satisfy $\delta(y_{i+1},x_{i}) < r_{i}$. Player I wins the game if $y_{n} \in U$. Also, let $\Gamma_{0}(y_{0})$ be a trivial game where no moves are being made, and where player I wins if $y_{0} \in U$.

\begin{lemma}
Player I has a winning strategy in $\Gamma_{n}(y_{0})$ if and only if $y_{0} \in U_{*}^{n}$.
\end{lemma}
\begin{proof}
The proof is by the induction on $n$.

For $n = 0$ the result is obvious. Suppose $n \in \mathbb{N}$ is such that the result is true for the game $\Gamma_{n}(y_{0})$, for each $y_{0}$. Consider the game $\Gamma_{n+1}(y_{0})$.

Suppose first that player I has a winning strategy in $\Gamma_{n + 1}(y_{0})$. Let $(x_{0},r_{0})$ be player I's initial move. The legal moves of player II following $(x_{0},r_{0})$ are points $y_{1}$ in the ball $O(x_{0},r_{0})$. Take a legal move $y_{1}$ of player II. The game that ensues after the move $y_{1}$ is essentially equivalent to $\Gamma_{n}(y_{1})$, and player I has a winning strategy there. Hence the induction hypothesis implies that $y_{1} \in U_{*}^{n}$. We have shown that $O(x_{0},r_{0}) \subseteq U_{*}^{n}$. On the other hand, $y_{0} \in O(x_{0},2r_{0})$. We conclude that $y_{0} \in U_{*}^{n+1}$.

Conversely, suppose that $y_{0} \in U_{*}^{n+1}$. Then there is a $x_{0}$ such that $y_{0} \in O(x_{0},2r_{0})$ and $O(x_{0},r_{0}) \subseteq U_{*}^{n}$. Notice that $(x_{0},r_{0})$ is a legal move by player I in $\Gamma_{n + 1}(y_{0})$. Legal moves $y_{1}$ of player II are restricted to the ball $O(x_{0},r_{0})$, and hence by the induction hypothesis, player I has a winning strategy after each such move $y_{1}$. We conclude that player I has a winning strategy in $\Gamma_{n + 1}(y_{0})$.
\end{proof}

We define the \textit{doubling distance} between open subsets $U$ and $V$ of $X$ as
\[d(U,V) = \max\left\{d_{\rightarrow}(U,V),d_{\rightarrow}(V,U)\right\}\,.\]
It is finite whenever $U$ and $V$ are non--empty bounded open sets. We let $\mathcal{U}_{X}$ denote the collection of all non--empty bounded open subsets of $X$. The following is a consequence of the part \eqref{u3} of Lemma \ref{thm.union}.

\begin{lemma}
The function $d$ is a metric on $\mathcal{U}_{X}$.
\end{lemma}

The next lemma is a variant of the familiar $5R$-covering theorem.

\begin{lemma}\label{7R}
Suppose that $X$ is separable. Let $U$ be a non-empty open bounded subset of $X$. There exists a countable collection $\mathcal{C}  = \{O(x_{\alpha},r_{\alpha})\}$ of pairwise disjoint open balls contained in $U$ such that $U_{*} \subseteq \bigcup_{\alpha} O(x_{\alpha},7r_{\alpha})$. If, moreover, $\delta$ is an ultrametric, or if $X = \mathbb{R}$, then the collection $\mathcal{C}$ can be chosen so that $U_{*} \subseteq \bigcup_{\alpha} O(x_{\alpha},2r_{\alpha})$.
\end{lemma}
\begin{proof}
We define a sequence of subsets $U_{\alpha}$ of $U$, indexed by ordinals, by recursion as follows: Let $U_{0} = U$. For a successor ordinal $\alpha + 1$, proceed thus: Suppose that $U_{\alpha}$ has been defined. If it has empty interior, the recursion ends. If its interior is not empty, we let $\rho_{\alpha} = \sup\{r: \text{there exists an }x\text{ such that }O(x,r) \subseteq U_{\alpha}\}$. Choose any point $x_{\alpha}$ and a number $r_{\alpha}$ such that $\tfrac{1}{2}\rho_{\alpha} \leq r_{\alpha}$ and $O(x_{\alpha},r_{\alpha}) \subseteq U_{\alpha}$. Let $U_{\alpha + 1} = U_{\alpha} \setminus O(x_{\alpha},r_{\alpha})$. For a limit ordinal $\lambda$, define $U_{\lambda} = \bigcap_{\alpha < \lambda} U_{\alpha}$.

Notice that for some countable ordinal $\xi$ the interior of $U_{\xi}$ is empty, for otherwise $X$ would contain an uncountable family of disjoint open balls, contradicting separability. Thus the collection $\mathcal{C}  = \{O(x_{\alpha},r_{\alpha}): \alpha < \xi\}$ is countable.

Now take any open ball $O(x,r)$ contained in $U$, and let $\alpha < \xi$ be the largest ordinal such that $O(x,r) \subseteq U_{\alpha}$ (notice that $\alpha$ is well defined by our specification of the limit step). Then $r \leq \rho_{\alpha}$, and hence $\tfrac{1}{2}r \leq r_{\alpha}$. Since $O(x,r)$ is not contained in $U_{\alpha+1}$, it meets the ball $O(x_{\alpha},r_{\alpha})$. Consequently, $O(x,2r) \subseteq O(x_{\alpha},7r_{\alpha})$. This proves the first claim of the lemma.

If $\delta$ is an ultrametric, or if $X = \mathbb{R}$, we can take $\mathcal{C}$ to be the collection of open balls contained in $U$ that are maximal with respect to set inclusion. Separability of $X$ guarantees that $\mathcal{C}$ is countable. Moreover, $U = \bigcup \mathcal{C}$, and consequently $U_{*} = \bigcup_{\alpha} O(x_{\alpha},2r_{\alpha})$.
\end{proof}

The following result provides the first connection between the doubling distance and the doubling measures.

\begin{thm}\label{7Rcor}
Let $U$ be an open subset of $X$, and $V$ a subset of $X$ with $d_{\rightarrow}(U,V) < \infty$. Let $1 \leq C < \infty$ and $\mu \in\mathcal{D}_C(X)$. Then
\[\mu(U) \geq C^{-3d_{\rightarrow}(U,V)}\mu(V).\]
If, in addition, $\delta$ is an ultrametric, or if $X = \mathbb{R}$, then
\[\mu(U) \geq C^{-d_{\rightarrow}(U,V)}\mu(V).\]
\end{thm}
\begin{proof}
We prove the first statement of the theorem. The second statement is proven similarly.
First we argue that for each open subset $U$ of $X$ it holds that
\begin{equation}\label{eqn.mupred}
\mu(U) \geq C^{-3} \mu(U_{*}).
\end{equation}

Suppose first that $U$ is a non-empty bounded open subset of $X$. Since $X$ carries a doubling measure, it is separable. Letting the collection $\mathcal{C}$ be as in Lemma \ref{7R} and using \eqref{eqn.doubling2} we obtain:
\[\mu(U_{*}) \leq \sum_{\alpha}\mu(O(x_{\alpha},7r_{\alpha})) \leq \sum_{\alpha}C^{3}\mu(O(x_{\alpha},r_{\alpha})) \leq C^{3}\mu(U).\]

Now let $U$ be a non-empty open subset of $X$. Take any $x \in U$ and notice that $U_{*}$ is a limit of a nested sequence of sets $(U \cap O(x,1))_{*} \subseteq (U \cap O(x,2))_{*} \subseteq \cdots$. For each $n \in \mathbb{N}$ we have $\mu(U \cap O(x,n)) \geq C^{-3}\mu((U \cap O(x,n))_{*})$. Taking the limit as $n \to \infty$ yields \eqref{eqn.mupred}.

Now let $d_{\rightarrow}(U,V) = k$. Let the sequence $U_{*}^{0}, U_{*}^{1},\dots$ be as in the definition of the directed doubling distance. By \eqref{eqn.mupred} we have $\mu(U_{*}^{n}) \geq C^{-3} \mu(U_{*}^{n+1})$ for each $n \in \mathbb{N}$. Since $U_{*}^{k} \supseteq V$, we obtain $\mu(U) = \mu(U_{*}^{0}) \geq C^{-3k}\mu(U_{*}^{k}) \geq C^{-3k}\mu(V)$, as desired.
\end{proof}

\subsection{A distance defined using doubling measures}
Let $X$ be a metric space that carries at least one doubling measure. Recall that a Borel set $T\subseteq X $ is called \emph{thin} or \emph{thin for doubling measures} if $\mu(T)=0$ for all $\mu\in\mathcal{D}(X)$.

For two Borel sets $U,V\subseteq X$, we define
$$m_{\rightarrow}(U,V) = \inf\left\{t\ge 0:\begin{array}{c}\text{For each }1 \leq C < \infty \text{ and each }\mu \in \mathcal{D}_{C}(X),\\ \mu(U) \geq C^{-t} \mu(V)\end{array}\right\}$$
(where the infimum of the empty set is assumed to be $\infty$). Also let $m(U,V) = \linebreak \max\{m_{\rightarrow}(U,V),m_{\rightarrow}(V,U)\}$.

It is easy to check that both $m_{\rightarrow}$ and $m$ satisfy the triangle inequality on the entire collection of Borel subsets of $X$. In general, it is perfectly possible that $m(U,V) = \infty$. This happens, for example, if there is some doubling measure $\mu$ such that $\mu(U) = 0$ while $\mu(V)>0$. However, if $U$ and $V$ are non-empty bounded open sets, then $m(U,V)$ is finite. We state the following simple result for the sake of completeness.

\begin{lemma}
The function $m$ is a pseudometric on $\mathcal{U}_{X}$.
\end{lemma}

\begin{rem}\rm
The function $m$ is not a genuine metric. For take two Borel subsets of $X$. If the symmetric difference $U \Delta V=(U \setminus V)\cup (V \setminus U)$ is a thin set, then $m(U,V) = 0$. The converse of this statement is also true\footnote{We are grateful to the reviewer for pointing out this equivalence, and for supplying the proof.}: if $m(U,V) = 0$, then $U \Delta V$ is thin. To see this, let $E_{1} =  U \setminus V$ and $E_{2} = V \setminus U$. If $m(U,V) = 0$, then for each doubling measure $\mu$ it holds that $\mu(U) = \mu(V)$ and consequently that $\mu(E_{1}) = \mu(E_{2})$. Suppose that $\mu(E_{1}) = \mu(E_{2}) > 0$ for some doubling measure $\mu$. Define $\nu(B) = 2\mu(B \cap E_{1}) + \mu(B \setminus E_{1})$ for each Borel subset $B$ of $X$. If $\mu$ is a $C$-doubling measure, $\nu$ is a $2C$-doubling measure. Moreover, $\nu(E_{1}) = 2\mu(E_{1}) > \mu(E_{1}) = \mu(E_{2}) = \nu(E_{2})$, yielding a contradiction.

One could obtain a metric from $m$ in the usual way, by identifying the sets in $\mathcal{U}_{X}$ that differ by a thin set.
\end{rem}

We now state a basic connection between the doubling distance $d$ and the pseudometric $m$. The result follows directly from Theorem \ref{7Rcor}.

\begin{cor}\label{thm.distancemeasure}
Let $X$ be a metric space that carries a doubling measure. Then for all open subsets $U$, $V$ of $X$ it holds that $m(U,V) \leq 3 d(U,V)$. If, in addition, $\delta$ is an ultrametric, or if $X = \mathbb{R}$, then $m(U,V) \leq d(U,V)$.
\end{cor}

\section{When are the two distances comparable?}\label{sec:main}
Having observed that the estimate $m(U,V)\lesssim d(U,V)$ holds for all open sets, we will next turn to the main theoretical result of the paper by investigating when $d(U,V)$ and $m(U,V)$ are actually comparable. It is quite easy to see that, in general, $d(U,V)$ can be much larger than $m(U,V)$. As we have previously remarked, if $U$ and $V$ differ by a thin set, then $m(U,V)=0$, but the doubling distance $d(U,V)$ may still be quite large. In fact, as the next example illustrates, one can construct non-empty bounded open subsets $U \subseteq V$ of the real line such that $V \setminus U$ is countable, while the doubling distance $d(U,V)$ is arbitrarily large.

\begin{exl}\label{exl.realline}\rm
Let $X = \mathbb{R}$. Take an $M \in \N$. Starting from an open interval $V \subset \R$, we construct an open set $U \subseteq V$ such that $V \setminus U$ is countable, and $d(U,V)= M$.

To begin, we define an operation $-$ on the collection $\mathcal{U}_{\mathbb{R}}$, which serves as a left inverse of the predecessor operation. First take an interval $(a,b)$. Define a decreasing sequence $y_{0},y_{1},\dots$ in $(a,b)$ converging to $a$ and an increasing sequence $x_{0},x_{1},\dots$ in $(a,b)$ converging to $b$ as follows:
	\begin{align*}
	y_{0} &= \tfrac{1}{4}(3a+b)\text{ and }y_{n+1} = \tfrac{1}{3}(2a + y_{n})\\
	x_{0} &= \tfrac{1}{4}(3b+a)\text{ and }x_{n+1} = \tfrac{1}{3}(2b + x_{n}).
	\end{align*}
We define $(a,b)_{-}$ to be the set $(a,b)$ with all the points $x_{n}$ and $y_{n}$ removed. Equivalently, $(a,b)_{-}$ is a union of the intervals $(y_{0},x_{0})$, $(y_{n+1},y_{n})$, and $(x_{n},x_{n+1})$. The points $y_{0}$ and $x_{0}$ have been chosen so that $(y_{0},x_{0})_{*} = (a,b)$. The rest of the sequences have been chosen so that the left endpoint of each interval $(y_{n+1},y_{n})_{*}$ is $a$ and the right endpoint of each interval $(x_{n},x_{n+1})_{*}$ is $b$. This shows that $((a,b)_{-})_{*} = (a,b)$.
	
Now take a set $U \in \mathcal{U}_{\mathbb{R}}$. Then $U$ can be written, uniquely, as a union of countably many disjoint open intervals, say $(a_{i},b_{i})$, for $i \in \mathbb{N}$. We define $U_{-} = \bigcup_{i \in \mathbb{N}}(a_{i},b_{i})_{-}$. Then $(U_{-})_{*} = U$.
	
Now let $V$ be an open interval and $U = V_{-}^{M}$ be the set obtained from $V$ by an $M$-fold application of the operation $-$. Thus $d(U,V) = M$. But since $V \setminus U$ is a countable set we have $m(U,V) = 0$.
\end{exl}

\subsection{Statement of the main result}
A non-empty subset of a metric space $X$ is called a \textit{simple open set} if it can be written as a finite union of open balls. According to our main result, if $V$ and $U$ are simple open sets, then $d(U,V)$ is comparable to $m(U,V)$.

\begin{thm}\label{thm.simple}
Let $X$ be a metric space that carries a doubling measure. If $U$ and $V$ are simple open sets, then $d(U,V) \leq 4[6m(U,V) + 2]$.
\end{thm}

\begin{rems}\rm\label{rem.closure}
\noindent\emph{(a)} If $U$ is a simple open set, then its closure ${\rm cl}\,U$ is a subset of $U_{*}$. This observation plays an important role in the proof. It is not true of general bounded open sets, as Example \ref{exl.realline} shows. Moreover, this fact implies that if $U$ and $V$ are simple open sets such that $U\neq V$ but  $m(U,V) = 0$, then $d(U,V)=1$. This also explains the need for an additive term in the bound of Theorem \ref{thm.simple}.\medskip

\noindent\emph{(b)} The assumption that $U$ and $V$ are both simple open sets may seem quite restrictive at first. On the other hand, the principal application of Theorem \ref{thm.simple} is Theorem \ref{thm.F.equiv}. For the purpose of establishing the latter result, Theorem \ref{thm.simple}, restrictive as it may seem, turns out to be sufficient.\medskip

\noindent\emph{(c)} One could extend the applicability of Theorem
\ref{thm.simple} further using the following observation: suppose $U,V\subseteq X$ are non-empty bounded open sets. Suppose one can find simple open sets $U',V'$ such that the doubling distance between $U$ and $U'$ and that between $V$ and $V'$ is small. Then the triangle inequality for $d$ and $m$ along with Theorems \ref{thm.distancemeasure} and \ref{thm.simple} allows one to obtain an upper bound on $d(U,V)$ in terms of $m(U,V)$. See Corollary \ref{cor} below for a more precise formulation.\medskip

\noindent\emph{(d)} We haven't paid too much attention to optimizing the constants in Theorems \ref{thm.distancemeasure} and \ref{thm.simple}. They can be certainly improved. 	
\end{rems}

For $U \in \mathcal{U}_{X}$ let
$$C_U = \min\{d(U,U')\,:\,U'\text{ is a simple open set}\}.$$
\begin{cor}\label{cor}
Let $X$ be a metric space that carries a doubling measure. For $U,V \in \mathcal{U}_{X}$, it holds that
\[d(U,V)\le 73C_U+73C_V+ 4[6m(U,V) + 2]\,.\]
\end{cor}

\begin{proof}
	Pick simple open sets $U'$ and $V'$ as in the definition of $C_U$, $C_V$.
	Applying the triangle inequality for $m$ together with Theorem \ref{thm.distancemeasure}, we have
	\begin{align*}
	m(U',V')\le m(U,V)+3C_U+3C_V\,.
	\end{align*}
	 Applying the triangle inequality for $d$ and Theorem \ref{thm.simple} for $U'$ and $V'$ then gives
	\begin{align*}
	d(U,V)\le C_U+C_V+d(U',V')\le C_U+C_V+4[6m(U',V')+2]\,.
	\end{align*}
	Putting the two estimates together yields the claim.
\end{proof}

\subsection{The proof of Theorem \ref{thm.simple}}
We prove that for simple open sets $U,V\subseteq X$,
\[d_\rightarrow(U,V)\le 4[6m_\rightarrow(U,V)+2]\,.\]

The proof is divided into two parts. In the first part we construct a sequence of sets $U = W_{0}, W_{1},\dots$ that serve as a proxy for the iterated predecessors of $U$, along with finite partitions $\mathcal{S}_{m}$ of $W_{m}$. We remark that the first part of the proof becomes trivial if $X$ is the real line with its usual metric, or if it is an ultrametric space. In these cases we could take $W_{m}$ to be $U_{*}^{m}$, and $\mathcal{S}_{m}$ to be the coarsest partition of $U_{*}^{m}$ into open balls.

In the second part of the proof we construct a measure that witnesses the desired lower bound on $m(U,V)$ in terms of $d(U,V)$. Starting with any given doubling measure $\lambda$ on $X$, we modify it step by step, as follows: At step $m$, we squeeze the measure inside $W_{m}$ obtained thus far by a factor $\epsilon$, in such a way that the measure of each element of the partition $\mathcal{S}_{m+1}$ of the set $W_{m+1}$ remains unchanged. This guarantees that the resulting measures are $\epsilon^{-6}$-doubling for any $m$. After $M$ steps, where $M$ is suitably chosen, the resulting measure of the set $U$ behaves as $\epsilon^{M}$, while the measure of $V$ remains bounded below by a positive constant independent of $\epsilon$. Letting $\epsilon$ be sufficiently small then completes the proof.\bigskip

\noindent\textsc{Part I: Constructing a sequence of sets.} Let $U$ be a simple open set. We construct a sequence of pairs $(W_{0}, \mathcal{S}_{0}), (W_{1},\mathcal{S}_{1}),\dots,$ where $W_{0} = U$, and for each $m \in \mathbb{N}$, $W_{m}$ is an open subset of $X$ and $\mathcal{S}_{m}$ is a partition of $W_{m}$ into finitely many, say $n_{m}$, non-empty Borel sets. This sequence is required to satisfy the following:
\begin{enumerate}
\renewcommand{\labelenumi}{(P\arabic{enumi})}
\renewcommand{\theenumi}{P\arabic{enumi}}
\item\label{P1} For each $S \in \mathcal{S}_{m}$ there is an open ball $O$ such that $S \subseteq O \subseteq W_{m}$.
\item\label{P2} $(W_{m})_{*} \subseteq W_{m+1} \subseteq (W_{m})_{*}^{4}$.
\item\label{P3} For each $S' \in \mathcal{S}_{m}$ there is an $S \in \mathcal{S}_{m+1}$ such that $S' \subseteq S$; for each $S \in \mathcal{S}_{m+1}$ there is an $S' \in \mathcal{S}_{m}$ such that $S' \subseteq S$.
\item\label{P4} For each $S \in \mathcal{S}_{m+1}$, the set $S \setminus W_{m}$ has a non-empty interior.
\end{enumerate}
Notice that \eqref{P1} implies that $W_{m}$ is a simple open set, while \eqref{P3} implies that $n_{m+1} \leq n_{m}$.

Let $n_{0} \in \mathbb{N}$ be the smallest number such that $U$ can be written as a union of $n_{0}$ open balls, say $O_{1}, \dots, O_{n_{0}}$. Define $W_{0} = U$. Let the partition $\mathcal{S}_{0}$ of $W_{0}$ consist of the sets $S_{0,i} = O_{i} \setminus \bigcup_{j < i}O_{j}$ for $i = 1, \dots, n_{0}$. The choice of $n_{0}$ ensures that the sets in $\mathcal{S}_{0}$ are all non-empty. Clearly $\mathcal{S}_{0}$ satisfies \eqref{P1}.

Suppose that $m \in \mathbb{N}$ is such that we have defined the set $W_{m}$ and the partition $\mathcal{S}_{m} = \{S_{m,1},\dots,S_{m,n_{m}}\}$ satisfying \eqref{P1}. For each $i \in \{1,\dots, n_{m}\}$ define
$$\rho_{i} = \sup\{r: \text{there is an }\linebreak x\text{ such that }S_{m,i} \subseteq O(x,r) \subseteq (W_{m})_{*}\}.$$ The supremum is taken over a non-empty set since $\mathcal{S}_{m}$ satisfies \eqref{P1}. The supremum is finite since $W_{m}$, and hence also $(W_{m})_{*}$, is bounded. Take some $x_{i}$ and $r_{i} > \rho_{i}/2$ such that $S_{m,i} \subseteq O(x_{i},r_{i}) \subseteq (W_{m})_{*}$.

First we argue that
\begin{equation}\label{eqn.sequ0}
(W_{m})_{*} \subseteq \textstyle\bigcup_{i = 1}^{n_{m}}O(x_{i},3r_{i}).
\end{equation}
Take any open ball $O(x,r) \subseteq W_{m}$. We have $x \in S_{m,i} \subseteq O(x_{i},r_{i})$ for some $i \in \{1,\dots,n_{m}\}$. If it were the case that $r_{i} < r$, we would have $\rho_{i} < 2r$, along with the inclusions $S_{m,i} \subseteq O(x_{i},r_{i}) \subseteq O(x,2r) \subseteq (W_{m})_{*}$, contradicting the definition of $\rho_{i}$. Thus $r \leq r_{i}$, and therefore $O(x,2r) \subseteq O(x_{i},3r_{i})$. This shows \eqref{eqn.sequ0}.

Secondly, for each $i \in \{1,\dots,n_{m}\}$, since $O(x_{i},r_{i})$ is contained in $(W_{m})_{*}$, we have
\begin{equation}\label{eqn.sequ1}
\textstyle\bigcup_{i = 1}^{n_{m}}O(x_{i},7r_{i}) \subseteq (W_{m})_{*}^{4}.
\end{equation}

Take $i \in \{1,\dots,n_{m}\}$. Since $2r_{i}> \rho_{i}$, the set $O(x_{i},2r_i) \setminus (W_{m})_{*}$ is not empty. Consequently, the set $B(x_{i},2r_{i}) \setminus (W_{m})_{*}$ is not empty. Now since $W_{m}$ is simple, ${\rm cl}\, W_{m} \subseteq (W_{m})_{*}$. Thus the set $B(x_{i},2r_{i}) \setminus {\rm cl}\,W_{m}$ is not empty. Take a point $y_{i} \in B(x_{i},2r_{i}) \setminus {\rm cl}\,W_{m}$. Let $n_{m+1}$ denote the number of distinct points among $y_{1},\dots,y_{n_{m}}$. Enumerate these points as $z_{1},\dots,z_{n_{m+1}}$. For each $k \in \{1,\dots,n_{m+1}\}$ let $I_{k} = \{i \in \{1,\dots,n_{m}\}: y_{i} = z_{k}\}$. Also let $i_{k}$ denote any element of $I_{k}$ maximizing $r_{i}$ over $i \in I_{k}$.

Take a $k \in \{1,\dots,n_{m+1}\}$. Since for each $i \in I_{k}$ the ball $B(x_{i},2r_{i})$ has a point in common (namely $z_{i}$) with the ball $B(x_{i_{k}}, 2r_{i_{k}})$, and since $r_{i} \leq r_{i_{k}}$, we have \begin{equation}\label{eqn.sequ2}
O(x_{i},3r_{i}) \subseteq O(x_{i_k},7r_{i_k})\qquad\text{ for each }i \in I_{k}.
\end{equation}

Define
$$W_{m+1} = \textstyle\bigcup_{k = 1}^{n_{m+1}}O(x_{i_{k}},7r_{i_{k}}).$$
That $W_{m+1}$ satisfies the first inclusion of condition \eqref{P2} follows from \eqref{eqn.sequ0} and \eqref{eqn.sequ2}. That it satisfies the second inclusion of \eqref{P2} follows from \eqref{eqn.sequ1}.

We proceed to define the partition $\mathcal{S}_{m+1}$ of $W_{m+1}$.

Since the points $z_{1},\dots,z_{n_{m+1}}$ are all distinct, and each $z_{k}$ is a point of $O(x_{i_{k}},7r_{i_{k}}) \setminus {\rm cl}\,W_{m}$, we can choose a $\delta > 0$ small enough that
\begin{itemize}
\item[{\rm(a)}] $O(z_{k},\delta) \subseteq O(x_{i_{k}},7r_{i_{k}})$,
\item[{\rm(b)}] $O(z_{k},\delta)$ is disjoint from ${\rm cl}\,W_{m}$ and
\item[{\rm(c)}] the balls $O(z_{1},\delta), \dots, O(z_{n_{m+1}},\delta)$ are pairwise disjoint.
\end{itemize}

We define $\mathcal{S}_{m+1}$ to consist of sets $S_{m+1,1},\dots,S_{m+1,n_{m+1}}$. For each $k \in \{1,\dots,n_{m+1}\}$ the set $S_{m+1,k}$ is built as follows: we first include the sets $S_{m,i}$ for all $i \in I_{k}$, and the ball $O(z_{k},\delta)$. Next we add a part of the ball $O(x_{i_{k}},7r_{i_{k}})$ chosen so that $S_{m+1,k}$ is disjoint from the other elements of $\mathcal{S}_{m+1}$. More rigorously, we let for $k \in \{1,\dots,n_{m+1}\}$,
$$S_{m+1,k}' = \textstyle\bigcup_{i \in I_{k}}S_{m,i} \textstyle\bigcup O(z_{k},\delta).$$
Next we define
$$D_{k} = O(x_{i_{k}},7r_{i_{k}}) \setminus \textstyle\bigcup_{j < k}O(x_{i_{j}},7r_{i_{j}}).$$
Finally we let $S_{m+1}' = S_{m+1,1}' \cup \cdots \cup S_{m+1,n_{m+1}}'$ and
$$S_{m+1,k} = S_{m+1,k}' \cup (D_{k} \setminus S_{m+1}').$$
Then $W_{m+1}  =\bigcup\mathcal{S}_{m+1}$. It remains to verify that $\mathcal{S}_{m+1}$ satisfies conditions \eqref{P1}, \eqref{P3}, and \eqref{P4}. Condition \eqref{P1} is satisfied since $S_{m+1,k} \subseteq O(x_{i_{k}},7r_{i_{k}}) \subseteq W_{m+1}$. Condition \eqref{P3} holds since $S_{m,i} \subseteq S_{m+1,k}$ for each $i \in I_{k}$ and $k \in \{1,\dots,n_{m+1}\}$. To verify condition \eqref{P4}, notice that $S_{m+1,k} \setminus W_{m}$ contains the open ball $O(z_{k},\delta)$.
This concludes part I.\medskip

\noindent\textsc{Part II: Constructing a measure.} Let $U$ and $V$ be simple open sets, and for each $m \in \mathbb{N}$ let $W_{m}$ and $\mathcal{S}_{m}$ be as in Part I of the proof (corresponding to $U$).

If $V\subseteq W_1$, it follows from \eqref{P2} that $d_\rightarrow (U,V)\le 4\le 4[6 m_\rightarrow(U,V)+2]$. So we may assume that $V \setminus W_1$ is non-empty. Let $M$ denote the least natural number such that $V \subseteq W_{M+2}$. As follows from \eqref{P2}, $W_{m} \subseteq U_{*}^{4m}$ for each $m \in \mathbb{N}$, hence $d_{\rightarrow}(U,V) \leq 4(M+2)$. We show that for all small enough $\epsilon>0$, there exists an $\epsilon^{-6}$-doubling measure $\mu_{\epsilon}$ on $X$ such that $\mu_{\epsilon}(U) \lesssim \epsilon^{M}\mu_{\epsilon}(V)$. This implies $M \leq 6m(U,V)$, leading to the desired estimate.

To that end, let us begin by choosing any doubling measure $\lambda$ on $X$. Let $\epsilon \in (0,1)$ and $m \in \mathbb{N}$. For each $S \in \mathcal{S}_{m+1}$ let $J_{m}(S)$ denote the union of the sets $S' \in \mathcal{S}_{m}$ contained in $S$. Define a Borel function $f_{m+1,\epsilon} : X \rightarrow \mathbb{R}$ as follows:
$$f_{m+1,\epsilon}(x)  = \begin{cases}1&\text{if }x \in X \setminus W_{m+1}\\\epsilon&\text{if }x \in W_{m}\\\displaystyle\frac{\lambda(S) - \epsilon \lambda(J_{m}(S))}{\lambda(S \setminus J_{m}(S))}&\text{if }x \in S \setminus J_{m}(S)\text{ for some }S \in \mathcal{S}_{m+1}.\end{cases}$$
Notice that $S \setminus J_{m}(S) = S \setminus W_{m}$ has a positive $\lambda$-measure by condition \eqref{P4}.

Let $\mu_{0,\epsilon}=\lambda$ and for $m\in\{1,\ldots, M\}$ define a measure $\mu_{m,\epsilon}$ recursively by the formula
$$\mu_{m,\epsilon}(E) = \int_{E}f_{m,\epsilon}d\mu_{m-1,\epsilon}\text{ for all Borel sets }E\,.$$
An easy induction shows that
\begin{align}
\mu_{m,\epsilon}(E) &= \lambda(E)\qquad\text{ for each }E \subseteq X \setminus W_{m}\label{eqn.preserve0}\\
\mu_{m,\epsilon}(S) &= \lambda(S)\qquad\text{ for each }S \in \mathcal{S}_{m}.\label{eqn.preserve1}
\end{align}

Let
$$K = \max\left\{\frac{\lambda(S)}{\lambda(S \setminus J_{m}(S))}:m \in \{0,\dots,M-1\},\, S \in \mathcal{S}_{m+1}\right\}.$$
Notice that this constant $K$ is independent of $\epsilon$ and that the functions $f_{1,\epsilon},\dots,f_{M,\epsilon}$ are bounded above by $K$.

Let $C$ be the doubling constant of $\lambda$. Take an $\epsilon < C^{-4}K^{-5}$. We next argue that $\mu_{M,\epsilon}$ is doubling with the doubling constant $\epsilon^{-6}$.

Take $x \in X$ and $r > 0$. Define a number $m \in \{0,\dots,M\}$ as follows: If $O(x,2r) \subseteq W_{M}$, let $m$ be the smallest number such that $O(x,2r) \subseteq W_{m}$; and if $O(x,2r)$ is not contained in $W_{M}$, let $m = M$. We estimate $\mu_{M,\epsilon}(O(x,r))$ from below and $\mu_{M,\epsilon}(O(x,2r))$ from above.

First note that
\begin{align}
\mu_{M,\epsilon}(O(x,r)) &= \epsilon^{M-m}\cdot\mu_{m,\epsilon}(O(x,r))
\label{eqn.below0}\\
\mu_{M,\epsilon}(O(x,2r)) &= \epsilon^{M-m}\cdot\mu_{m,\epsilon}(O(x,2r)).
\label{eqn.above0}
\end{align}
Indeed, if $m = M$, the equations are trivially true. And if $m < M$, they follow since the functions $f_{M,\epsilon},\dots,f_{m+1,\epsilon}$ are all identically $\epsilon$ on $W_{m}$.

First suppose that $m < 5$. Since each of the functions $f_{m,\epsilon},\dots,f_{1,\epsilon}$ are bounded below by $\epsilon$ and above by $K < 1/\epsilon$, we obtain
\begin{align*}
\mu_{m,\epsilon}(O(x,r)) &\geq \epsilon^{5}\cdot\lambda(O(x,r))\\
\mu_{m,\epsilon}(O(x,2r)) &\leq K^5 \cdot \lambda(O(x,2r)).
\end{align*}
This gives
$$\frac{\mu_{M,\epsilon}(O(x,r))}{\mu_{M,\epsilon}(O(x,2r))} =
\frac{\mu_{m,\epsilon}(O(x,r))}{\mu_{m,\epsilon}(O(x,2r))} \geq
\epsilon^{5}K^{-5} \frac{\lambda(O(x,r))}{\lambda(O(x,2r))} \geq
\epsilon^{5}K^{-5}C^{-1} \geq \epsilon^{6}.$$

Now we turn to the case $m \geq 5$. Since each of the functions $f_{m,\epsilon},\dots,f_{m-4,\epsilon}$ are bounded below by $\epsilon$ and above by $K < 1/\epsilon$, we obtain
\begin{align}
\mu_{m,\epsilon}(O(x,r)) &\geq \epsilon^{5}\cdot\mu_{m-5,\epsilon}(O(x,r))
\label{eqn.below1}\\
\mu_{m,\epsilon}(O(x,2r))&\leq K^5\cdot\mu_{m-5,\epsilon}(O(x,2r)).
\label{eqn.above1}
\end{align}

We next derive a lower bound on $\mu_{m-5,\epsilon}(O(x,r))$.
First we argue that if a ball $O(x,r/3)$ meets a set $S \in \mathcal{S}_{m-5}$, then $S \subseteq O(x,r)$. To see this,  suppose that $O(x,r/3)$ meets some $S \in \mathcal{S}_{m-5}$. By condition \eqref{P1}, there is an $O(x',r')$ such that $S \subseteq O(x',r') \subseteq W_{m-5}$. It suffices to show that $r'\le r/3$ since then $S \subseteq O(x',r') \subseteq O(x,r)$.
Suppose on the contrary that $r/3 < r'$. Then $O(x,r) \subseteq O(x',5r') \subseteq (W_{m-5})_{*}^{3}$, and by condition \eqref{P2},
 $(W_{m-5})_{*}^{3} \subseteq W_{m-2}$. Thus
 $O(x,r) \subseteq W_{m-2}$, so
 $O(x,2r) \subseteq (W_{m-2})_*\subset W_{m-1}$, contradicting the choice of $m$. Hence we have shown that $S\subseteq O(x,r)$. Using this inclusion for all $S\in\mathcal{S}_{m-5}$ intersecting $O(x,r/3)$, we arrive at the estimate
\begin{align*}
\mu_{m-5,\epsilon}(O(x,r) \cap W_{m-5}) &\geq \sum_{S \in \mathcal{S}_{m-5}:\,S \cap O(x,r/3) \neq \oslash} \mu_{m-5,\epsilon}(S)\\& = \sum_{S \in \mathcal{S}_{m-5}:\,S \cap O(x,r/3) \neq \oslash} \lambda(S)\\& \geq \lambda(O(x,\tfrac{1}{3}r)\cap W_{m-5}),
\end{align*}
where the equality in the second line follows by \eqref{eqn.preserve1}. On the other hand, in view of \eqref{eqn.preserve0} we have
$$\mu_{m-5,\epsilon}(O(x,r) \setminus W_{m-5}) = \lambda(O(x,r) \setminus W_{m-5}).$$
Combining these estimates we obtain
\begin{equation}\label{eqn.below2}
\mu_{m-5,\epsilon}(O(x,r)) \geq \lambda(O(x,r/3)).
\end{equation}

We next turn to an upper bound on $\mu_{m-5,\epsilon}(O(x,2r))$.
First we argue that if a ball $O(x,2r)$ meets a set $S \in \mathcal{S}_{m-5}$, then $S \subseteq O(x,3r)$. To see this,  suppose that $O(x,2r)$ meets some $S \in \mathcal{S}_{m-5}$. By condition \eqref{P1}, there is an $O(x',r')$ such that $S \subseteq O(x',r') \subseteq W_{m-5}$. Suppose first that $r/2 < r'$. Then $O(x,r) \subseteq O(x',7r') \subseteq (W_{m-5})_{*}^{3}  \subseteq W_{m-2}$. Thus $O(x,2r) \subseteq W_{m-1}$, contradicting the choice of $m$. Hence $r' \leq r/2$. But then
\[S\subseteq O(x',r') \subseteq O(x,3r)\,,\]
as desired. Using this for all $S\in\mathcal{S}_{m-5}$ we obtain the estimate
\begin{align*}
\mu_{m-5,\epsilon}(O(x,2r) \cap W_{m-5}) &\leq \sum_{S \in \mathcal{S}_{m-5}:\,S \cap O(x,2r) \neq \oslash} \mu_{m-5,\epsilon}(S)\\& = \sum_{S \in \mathcal{S}_{m-5}:\,S \cap O(x,2r) \neq \oslash} \lambda(S)\\& \leq \lambda(O(x,3r)\cap W_{m-5}).
\end{align*}
In view of \eqref{eqn.preserve0} we also have
$$\mu_{m-5,\epsilon}(O(x,2r) \setminus W_{m-5}) = \lambda(O(x,2r) \setminus W_{m-5}).$$
Combining these estimates we obtain
\begin{equation}\label{eqn.above2}
\mu_{m-5,\epsilon}(O(x,2r)) \leq \lambda(O(x,3r)).
\end{equation}

Thus, putting together the three pairs of estimates \eqref{eqn.below0}-\eqref{eqn.above0}, \eqref{eqn.below1}-\eqref{eqn.above1}, and \eqref{eqn.below2}-\eqref{eqn.above2}, we obtain
\[
\frac{\mu_{M,\epsilon}(O(x,r))}{\mu_{M,\epsilon}(O(x,2r))} \geq \epsilon^{5}K^{-5}\, \frac{\mu_{m-5,\epsilon}(O(x,r))}{\mu_{m-5,\epsilon}(O(x,2r))} \geq \epsilon^{5}K^{-5}\,\frac{\lambda(O(x,\tfrac{1}{3}r))}{\lambda(O(x,3r))} \geq \epsilon^{5}K^{-5}C^{-4}\,\geq \epsilon^{6}.
\]
We have thus shown that $\mu_{M,\epsilon}$ is $\epsilon^{-6}$-doubling.

We argue that the ratio
\begin{equation}\label{eqn.ratio}
\frac{\mu_{M,\epsilon}(U)}{\mu_{M,\epsilon}(V)\epsilon^{M}}
\end{equation}
is bounded above as $\epsilon\downarrow 0$. Since each of the functions $f_{1,\epsilon},\dots,f_{M,\epsilon}$ is identically $\epsilon$ on $W_{0} = U$, we have $\mu_{M,\epsilon}(U) = \epsilon^{M} \cdot \lambda(U)$. On the other hand, $\mu_{M,\epsilon}(V)$ is bounded below by a positive constant that is independent of $\epsilon$. Indeed, $V$ is not contained in $W_{M+1}$. Since $V \setminus W_{M+1} \subseteq V \setminus {\rm cl}\, W_{M} \subseteq V \setminus W_{M}$, the interior of the set $V \setminus W_{M}$ is not empty. Hence $\lambda(V \setminus W_{M}) > 0$. Therefore $\mu_{M,\epsilon}(V) \geq \mu_{M,\epsilon}(V \setminus W_{M}) = \lambda(V \setminus W_{M})$. The result follows.

Finally, we show that $M \leq 6m_{\rightarrow}(U,V)$. For suppose to the contrary. Take $t$ such that $m_{\rightarrow}(U,V) < t < \tfrac{1}{6}M$. Since $\mu_{M,\epsilon} \in \mathcal{D}_{\epsilon^{-6}}(X)$, the ratio \eqref{eqn.ratio} is bounded below by $\epsilon^{6t - M}$. However, the latter bound goes to infinity as $\epsilon \downarrow 0$. This contradicts the conclusion that \eqref{eqn.ratio} is bounded above. $\Box$\bigskip

\section{Lipschitz functions with respect to the doubling metric}\label{sec:maps}
In this section, we consider a continuous surjective function $f \colon X \to Y$ between two metric spaces $X$ and $Y$ and study the induced map $\Phi$, $U\mapsto f^{-1}(U)$ from the collection of non--empty open subsets of $Y$ to that of $X$. We also consider the pushforward measures $\mu\circ f^{-1}$ on $Y$ for measures $\mu \in \mathcal{D}(X)$. We are interested in functions $f$ satisfying some of the following conditions:

\begin{enumerate}
\renewcommand{\labelenumi}{(F\arabic{enumi})}
\renewcommand{\theenumi}{F\arabic{enumi}}
\item\label{F1} There exists a $K<\infty$ such that $m(f^{-1}(U),f^{-1}(V)) \leq K\,m(U,V)$ for all sets $U$ and $V$ in $\mathcal{U}_{Y}$.
\item\label{F2} There exists a $K<\infty$ such that for each $1 \leq C < \infty$ and each $\mu \in \mathcal{D}_{C}(X)$ it holds that $\mu \circ f^{-1} \in \mathcal{D}_{C^{K}}(Y)$.
\item\label{F3} There exists a $K<\infty$ such that $d(f^{-1}(U),f^{-1}(V)) \leq K\,d(U,V)$ for all sets $U$ and $V$ in $\mathcal{U}_{Y}$.
\end{enumerate}

Note that the condition \eqref{F3} simply states that the map $\Phi$ is Lipschitz with respect to the doubling metric on the collection $\mathcal{U}_{Y}$. Likewise, \eqref{F1} says that $\Phi$ is Lipschitz with respect to the pseudometric $m$ on $\mathcal{U}_{Y}$. The condition \eqref{F2} says that $f$ preserves doubling measures, quantitatively.

\begin{prop}\label{prop.F_cond}
Let $f \colon X \to Y$ be a continuous surjective function. Suppose that the space $X$ carries at least one doubling measure. Then the conditions \eqref{F1} and \eqref{F2} are equivalent, and both are implied by \eqref{F3}.
\end{prop}
\begin{proof}
The equivalence of \eqref{F1} and \eqref{F2} is a direct corollary of the definition of $m$.

Suppose that \eqref{F3} holds. Take a measure $\mu \in \mathcal{D}_C(X)$ and consider an open ball $O(y,r)$ in $Y$. Then $d_\rightarrow(f^{-1}(O(y,r)),f^{-1}(O(y,2r))) \le K$, which implies by Theorem \ref{7Rcor} that $\mu(f^{-1}(O(y,2r))) \le C^{3K}\mu(f^{-1}(O(y,r))$. Thus $\mu \circ f^{-1} \in \mathcal{D}_{C^{3K}}(Y)$. This shows that \eqref{F3}$\Longrightarrow$\eqref{F2}.
\end{proof}

\begin{rem}\label{thm.k3}\rm
Let us note that the implication \eqref{F3} $\Longrightarrow$ \eqref{F2} multiplies $K$ by 3 and that the  implications \eqref{F1} $\Longrightarrow$ \eqref{F2} and \eqref{F2} $\Longrightarrow$ \eqref{F1} both preserve $K$. If $\delta$ is an ultrametric, or if $X = \mathbb{R}$, then also the implication \eqref{F3} $\Longrightarrow$ \eqref{F2} preserves $K$.
\end{rem}

The following lemma simplifies the verification of condition \eqref{F3}: it suffices to check that the condition holds for pairs of concentric balls.
\begin{lemma}\label{thm.F3balls}
Let $f \colon X \to Y$ be a continuous surjective function. Then $f$ satisfies \eqref{F3} if and only if there exists a $K<\infty$ such that for each $y \in Y$ and each $r > 0$ the directed doubling distance between $f^{-1}(O(y,r))$ and $f^{-1}(O(y,2r))$ is at most $K$.
\end{lemma}
\begin{proof}
Let $K$ be as in the lemma. We argue that for $U \in \mathcal{U}_{Y}$ we have
\begin{equation}\label{eq:step_1}
d(f^{-1}(U),f^{-1}(U_{*})) \leq K.
\end{equation}
To see this, let $I$ be the set of pairs $(y,r)$ where $y \in Y$ and $r > 0$ such that $O(y,r) \subseteq U$. For each $\alpha = (y,r)$ in $I$, let $U_{\alpha} = f^{-1}(O(y,r))$ and $V_{\alpha} = f^{-1}(O(y,2r))$. Now we apply Lemma \ref{thm.union}: since $d_{\rightarrow}(U_{\alpha},V_{\alpha}) \leq K$ for each $\alpha \in I$, we have $d_{\rightarrow}(\bigcup_{\alpha \in I} U_{\alpha},\bigcup_{\alpha \in I}V_{\alpha}) \leq K$. But $\bigcup_{\alpha \in I}U_{\alpha} = f^{-1}(U)$, and $\bigcup_{\alpha \in I}V_{\alpha} = f^{-1}(U_{*})$.
	
As in the proof of Lemma \ref{thm.union}\eqref{u4}, applying the estimate \eqref{eq:step_1} inductively, we find that if $V\subseteq U_{*}^n$, then $d_{\rightarrow}(f^{-1}(U), f^{-1}(V))\le Kn$. This completes the verification of \eqref{F3}.	
\end{proof}

We proceed with some examples, starting with a well-known class of quasisymmetric homeomorphisms.

\begin{exl}\label{exl.quasisymmertic}\rm
A quasisymmetric homeomorphism $f \colon X \to Y$ between metric spaces $X$ and $Y$ satisfies all three conditions \eqref{F1}, \eqref{F2}, and \eqref{F3}. The fact that quasisymmetric homeomorphisms satisfy \eqref{F3} follows directly from the estimates in \cite[Proposition 1.2]{ORS2012}. Lemma 16.4 in \cite{DavidSemmes1997} essentially says that quasisymmetric homeomorphisms satisfy \eqref{F2}.

Now recall that the inverse of a quasisymmetric homeomorphism is also quasisymmetric. Hence, if $f$ is a quasisymmetric homeomorphism, then the induced map $\Phi$ is bi-Lipschitz on $\mathcal{U}_{Y}$ with respect to the doubling metric. We return to this connection in subsection 5.1 below.
\end{exl}

The examples below point out two features of the conditions (F1)--(F3). Firstly, these conditions may well apply to functions that are not necessarily injective.

\begin{exl}\rm
Let $X$ and $Y$ be non-empty metric spaces, and $X \times Y$ be equipped with the supremum metric. Then the projection $X \times Y\to X$, $(x,y)\mapsto x$, satisfies the condition \eqref{F3} with the constant $K = 1$, and hence (in view of Remark \ref{thm.k3}), it satisfies \eqref{F2} and \eqref{F1} with $K = 3$.
\end{exl}

\begin{exl}\label{exl.shift}\rm
Let $m\in\N$ and $\Lambda=\{0,\ldots,m\}$. Consider the space $\Sigma=\Lambda^\N$ with its usual ultrametric given as follows: $\delta(\iii,\jjj)=2^{-n}$, if $\iii=i_1i_2\ldots,\jjj=j_1j_2\ldots\in\Sigma$, and $n=n(\iii,\jjj)$ is the unique element of $\mathbb{N}$ such that $i_1\ldots i_n=j_1\ldots j_n$ and $i_{n+1}\neq j_{n+1}$. Now, if we consider the backward shift $\sigma\colon\Sigma\to\Sigma$, $\sigma(i_1i_2i_3\ldots)=i_2i_3\ldots$, then the condition \eqref{F1} holds for $\sigma$ with the constant $K=1$. Hence (in view of Remark \ref{thm.k3}), \eqref{F2} and \eqref{F1} also hold with $K = 1$.
\end{exl}

And secondly, even for an injective function $f$, it is perfectly possible that $f$ satisfies conditions (F1)--(F3) while its inverse violates them.

\begin{exl}\rm Let $X = \{x \in (0,1) \times [0,1): x_{2} = 0\text{ or }x_{1} = x_{2}\}$ and $Y = (0,1) \times \{0, 1\}$, equipped, for simplicity, with the supremum metric inherited from $\mathbb{R}^{2}$. Define a homeomorphism $f \colon X \to Y$ by letting $f(x) = (x_{1},0)$ if $x_{2} = 0$ and $f(x) = (x_{1},1)$ if $x_{2} > 0$. It is not difficult to see that $f$ satisfies \eqref{F3}. To see that $f^{-1}$ does not satisfy \eqref{F3}, consider the point $x = (2r,0)$ and the balls $U_{r} = O(x,r)$ and $V_{r} = O(x,2r)$ in $X$. Then the doubling distance between $f(U_{r})$ and $f(V_{r})$ in $Y$ goes to infinity as $r$ becomes small.
\end{exl}

\begin{exl}\rm Consider the setup of Example \ref{exl.shift}. A \textit{permutation homeomorphism} is a map $f = f_{r} \colon \Sigma \rightarrow \Sigma$ given by $f(i_{1},i_2,\dots) = (i_{r(1)},i_{r(2)},\dots)$ where $r \colon \N \rightarrow \N$ is a bijection. It can be checked that for such a permutation homeomorphism, the condition \eqref{F3} holds if and only if
\[\sup_{n\in\N}(n-r(n))<\infty\,.\]
If we choose the function $r$ so that $n-r(n)$ is bounded from above but not from below, then $f=f_r$ satisfies all the conditions \eqref{F1}-\eqref{F3}, but the inverse $f^{-1}$ fails to satisfy \eqref{F3} (in fact it also fails to satisfy \eqref{F1}-\eqref{F2}, see Theorem \ref{thm.F.equiv} below).	
\end{exl}

The main result of this section is the following.

\begin{thm}\label{thm.F.equiv}
Suppose that the metric spaces $X$ and $Y$ carry some doubling measures, and that each closed ball in $X$ is compact. Let $f : X \rightarrow Y$ be a continuous surjective function such that the preimage of each bounded subset of $Y$ under $f$ is a bounded subset of $X$.  Then $f$ satisfies condition \eqref{F1} if and only if it satisfies \eqref{F2}, if and only if it satisfies \eqref{F3}.
\end{thm}
\begin{proof}
Suppose that $f$ satisfies \eqref{F1}. We show that it satisfies \eqref{F3}. The other implications follow from Proposition \ref{prop.F_cond}.
	
Let $K$ be as in condition \eqref{F1}. We show that the condition of Lemma \ref{thm.F3balls} holds (with the constant $72K+8$).

Take $y \in Y$ and $r > 0$. Under our assumptions on $X$ and $f$, the set $f^{-1}\left(B(y,r/2)\right)$ is a compact subset of $X$. As it is contained in the open set $f^{-1}\left(O(y,r)\right)$, there exists a simple open set $U'$ such that
$$f^{-1}\left(B(y,\tfrac{1}{2}r)\right) \subseteq U' \subseteq f^{-1}\left(O(y,r)\right)\,.$$
Likewise, there is a simple open set $V'$ such that
$$f^{-1}\left(B(y,2r)\right) \subseteq V' \subseteq f^{-1}\left(O(y,4r)\right)\,.$$
Hence, using Theorem \ref{thm.simple}, we have
\begin{align*}
d\left(f^{-1}\left(O(y,r)\right),f^{-1}\left(O(y,2r)\right)\right) &\leq d(U',V')\\&\leq 24\, m(U',V')+8\\&\leq 24\, m\left(f^{-1}\left(O(y,\tfrac{1}{2}r)\right),f^{-1}\left(O(y,4r)\right)\right)+8\\& \leq 72K+8\,,
\end{align*}
as required.
\end{proof}

\section{Further remarks}\label{sec:remarks}
We complete the paper by discussing (a) an open problem on a possible connection between quasisymmetric functions and those that are bi-Lipschitz with respect to the doubling metric, (b) a definition of porosity based on the doubling metric.

\subsection{Bi-Lipschitz functions with respect to the doubling metric}
We remarked above (Example \ref{exl.quasisymmertic}) that if $f$ is a quasisymmetric homeomorphism, then both $f$ and its inverse satisfy \eqref{F3}, or, in other words, the induced map $\Phi$ is bi-Lipschitz with respect to the doubling metric. This leads us to ask under what conditions the converse is also true.

Below we discuss two simple examples. The first example shows that in general, a homeomorphism $f$ such that both $f$ and its inverse $f^{-1}$ satisfy \eqref{F3} need not be quasisymmetric.

\begin{exl}\rm
Let the metric spaces $X$ and $Y$ have the same underlying set $\{0,1\} \times \mathbb{N}$. Let the metric $\delta_{X}$ on $X$ be given by
$$\delta_{X}((i,n),(j,m)) = \begin{cases}0 & \text{if }i = j, n = m,\\1 & \text{otherwise}.\end{cases}$$
Let the metric $\delta_{Y}$ on $Y$ be given by
$$\delta_{Y}((i,n),(j,m)) = \begin{cases}0 & \text{if }i = j, n = m,\\2^{-n-1} & \text{if }i \neq j, n = m,\\1 & \text{otherwise}.\end{cases}$$
Let $f$ be the identity map.

That the identity map is not quasisymmetric is easy to see.

Notice that the doubling distance between any two non-empty sets in $X$ is at most $1$. We show that the doubling distance between any two non-empty sets in $Y$ is at most $2$. This would imply that both the identity maps $X \to Y$ and $Y \to X$ satisfy (F3).

Consider a non-empty set $U \subseteq Y$. Let $(i,n) \in U$. We have $O_{Y}((i,n),2^{-n-1}) = \{(i,n)\}$. Now, the ball $O_{Y}((i,n),2^{-n})$ equals the set $\{(0,n),(1,n)\}$, and it therefore coindices with the ball $O_{Y}((i,n),1)$. Finally $O_{Y}((i,n),2) = Y$. This shows that the doubling distance between the set $U$ and the whole space $Y$ is at most $2$. Hence the doubling distance between any two non-empty subsets of $Y$ is at most $2$, as desired.
\end{exl}

Of course, both spaces in the preceding example are discrete, and hence not uniformly perfect. A slight modification of the first example yields a homeomorphism that carries a uniformly perfect space into a space that is not uniformly perfect. This contrasts with the behavior of quasisymmetric maps: as is well-known, quasisymmetric maps preserve uniform perfection.

\begin{exl}\rm
We modify the preceding example as follows: Let the metric spaces $X$ and $Y$ have the same underlying set, $[0,1] \times \mathbb{N}$. Let the metric $\delta_{X}$ on $X$ be given by
$$\delta_{X}((a,n),(b,m)) = \begin{cases}|a-b| & \text{if } n = m,\\1 & \text{otherwise}.\end{cases}$$
Let the metric $\delta_{Y}$ on $Y$ be given by
$$\delta_{Y}((a,n),(b,m)) = \begin{cases}2^{-n}|a-b| & \text{if }n = m,\\2 & \text{otherwise}.\end{cases}$$
Thus $X$ is uniformly perfect, while $Y$ is not.

As before, let $f$ be the identity map.

We argue that the identity map $Y \to X$ satisfies (F3). Thus consider a point $(x,n)$ and $r > 0$. We are interested in computing the $d_{Y}$-distance between $U = O_{X}((x,n),r)$ and $V = O_{X}((x,n),2r)$. Consider the following three cases:
\begin{itemize}
\item $1 < r$. In this case $U = V = X$.
\item $r \leq 1 < 2r$. In this case $U$ is a subset of $[0,1] \times \{n\}$. Hence $U = O_{Y}((x,n),2^{-n}r)$. Moreover the ball $O_{Y}((x,n),2^{-n+1}r)$ equals all of the set $[0,1] \times \{n\}$, and it coincides with the ball $O_{Y}((x,n),2)$. And $O_{Y}((x,n),4) = Y$.
\item $2r \leq 1$. In this case both $U$ and $V$ are subsets of $[0,1] \times \{n\}$. Hence $U = \linebreak O_{Y}((x,n),2^{-n}r)$ and $V = O_{Y}((x,n),2^{-n+1}r)$.
\end{itemize}
Thus in all cases $d_{Y}(U,V) \leq 2$.

A similar argument shows that the identity map $X \to Y$ satisfies (F3).
\end{exl}

\noindent The examples help us fine-tune our question.\\\medskip

\noindent\textbf{Open problem:} \textit{Suppose that $X$ and $Y$ are uniformly perfect spaces, and $f : X \to Y$ a homeomorphism such that both $f$ and $f^{-1}$ satisfy \eqref{F3}. Is $f$ quasisymmetric?}
\medskip

We know that the answer to this question is affirmative in two special cases.

One is the case where $X = Y = \mathbb{R}$ with its usual metric. As is well-known (see e.g. \cite[Remark 13.20]{Heinonen2001}), if a homeomorphism $f \colon \mathbb{R} \to \mathbb{R}$ satisfies \eqref{F2} (or equivalently if it satisfies \eqref{F3}), then it is quasisymmetric. Notice that in this case we need not require that the inverse of $f$ satisfies \eqref{F3}.

The second special case is when $X$ and $Y$ are uniformly perfect ultrametric spaces. The proof of this fact makes use of the metric characterization of quasisymmetric maps (see e.g.  \cite{Koskela2009}). We omit the details.

\subsection{Porosity conditions}
The doubling distance may be used in a natural way to define various concepts of \emph{porosity}. Here, we explain one of the many possibilities. Consider a closed subset $S$ of $X$ and a point $x \in S$. Define the $d$-\textit{porosity index} of $S$ at $x$ as
\[{\rm por}_{d}(S,x) = \underset{\alpha\downarrow 0}{\lim}\inf\{d_{\rightarrow}\left(O(y,r)\setminus S,\{x\}\right):y \in X,\, 0 < r \leq \alpha,\, x \in O(y,r)\}\,.\]
We say that $S$ is $d$-porous, if the $d$-porosity index of $S$ at $x$ is finite for all $x\in S$.
	
Let us compare the $d$-porosity with some more classical notions of porosity. Recall that the \emph{upper porosity} of $S$ at $x\in S$ is defined by
\[{\rm por}(S,x)=\limsup_{r\downarrow 0}\left\{\frac{\varrho}{r}\,:\,O(y,\varrho)\subseteq O(x,r)\setminus S\text{ for some }y\in O(x,r-\varrho)\right\}\,.\]

The set $S$ is called \emph{upper porous} if ${\rm por}(S,x)>0$ for all $x\in S$. This is, perhaps, the most classical notion of porosity, and it has appeared in the literature under various names and with slightly varying definitions (see e.g. \cite{Zajicek1987}, \cite{Zajicek2005}).
	
If $S$ is upper porous, it is clear that it is also $d$-porous. Indeed, if $O(x,r)\setminus S$ contains a ball of radius $2^{-k} r$, then $d_{\rightarrow}(O(x,r)\setminus S, \{x\})\le k$. However, it is possible that $S$ is (uniformly) $d$-porous, even if it is not upper porous.  For instance consider
\[S = \{0\}\bigcup_{k\in\N}\left\{-\frac{1}{k+1},\frac1{k+1}\right\}\subseteq \R\,.\]
Then $d_{\rightarrow}(O\setminus S,\{x\}) \le 2$ for all $x\in S$ and all balls $O$ containing $x$. However, it is easy to see that ${\rm por}(S,0)=0$.
	
As the above discussion indicates,  $d$-porosity is, a priori, a weaker condition than upper porosity. Nevertheless, it can be shown that each $d$-porous set is $\sigma$-upper porous and thus $d$-porous sets share many properties with upper porous sets. For instance, any $d$-porous subset $S$ of $X$ is thin.
	
Let us briefly discuss a global variant of the $d$-porosity. Let us say that a bounded set $S \subseteq X$ is $(d,m_n)$-\textit{porous}, for a sequence $m_n \in \N$, if there are open sets $V_n\subseteq X$ disjoint from $S$ such that $\bigcup_{n}V_n$ is bounded, $\sum_{n \in \N}\textbf{1}[V_n]$ is bounded (for some $N$, each point in $X$ belongs to at most $N$ of the sets $V_n$) and $d_\rightarrow(V_n,S)\le m_n$ for all $n\in\N$. 

It can be shown that this is a generalization of the notion of $(\alpha_n)$-porosity introduced in \cite{CsornyeiSuomala2012}. Given a sequence $(\alpha_n)_{n \in \mathbb{N}}$ with $0 < \alpha_{n} < 1$, the set $S$ is said to be $(\alpha_n)$-porous if there is a constant $N \in \mathbb{N}$ and a sequence of (finite or countably infinite) coverings $\mathcal{B}_{n} = \{O(x_{n,j},r_{n,j})\}$ of $S$ by open balls with the following properties: each ball $O(x_{n,j},r_{n,j})$ contains a sub ball $O(y_{n,j},\alpha_{n}r_{n,j}) \subseteq O(x_{n,j},r_{n,j}) \setminus S$, each point of $X$ belong to at most $N$ different balls $O(y_{n,j},\alpha_{n}r_{n,j})$, and the set $\bigcup_{n,j}O(x_{n,j},r_{n,j})$ is bounded\footnote{This last assumption is implicit in \cite[Lemma 4.1]{CsornyeiSuomala2012}.}.

One can show that if $\alpha_n=2^{-m_n+1}$ with $m_{n} \geq 1$ and $S$ is $(\alpha_n)$-porous, then it is $(d,m_n)$-porous.
		
\begin{rem}\rm
Suppose that $S \subseteq X$ is compact, all closed balls $B \subseteq X$ are compact, $m \in \mathbb{N}$, and for all $x \in S$, the $d$-porosity index of $S$ at $x$ is at most $m$. Using elementary covering arguments, we may conclude that $S$ is $(d,m_n)$-porous with the constant sequence $m_n=m$. Consequently, under the above compactness assumptions, each $d$-porous set is a countable union of constant sequence $(d,m_n)$-porous sets.
\end{rem}
		
We provide a variant of \cite[Lemma 4.1]{CsornyeiSuomala2012} for $(d,m_n)$-porous sets.
	
\begin{prop}
If $\sum_{n\in\mathbb{N}}^\infty \varepsilon^{m_n}=\infty$ for all $\varepsilon>0$, then each $(d,m_n)$-porous set $S\subseteq X$ is thin.
\end{prop}
	
\begin{proof}
Let $\mu\in\mathcal{D}_C(X)$. Suppose that $S$ is $(d,m_n)$-porous and let $\sum_{n \in\mathbb{N}}\mathbf{1}[V_n]\le N$. Using  Theorem \ref{7Rcor}, we get
	 	
\[\sum_{n \in\mathbb{N}}\mu(S)C^{-3m_{n}} \leq \sum_{n \in\mathbb{N}} \mu(V_{n}) \leq N \mu(\bigcup_{n \in \mathbb{N}}V_{n})\,.\]
Since $\bigcup_n V_n$ is bounded, the right-hand side of the estimate is finite. On the other hand, $\sum_{n \in\mathbb{N}}C^{-3m_{n}} = \infty$. This is possible only if $\mu(S)=0$.
\end{proof}

\bibliographystyle{plain}
\bibliography{Doubling_measures14}
\end{document}